\documentclass{article}

\usepackage{amsmath}
\usepackage{amsfonts}
\usepackage{amssymb}
\usepackage{amsthm}
\usepackage{enumerate}

\newtheorem{theorem}{Theorem}[section]
\newtheorem{lemma}[theorem]{Lemma}
\newtheorem{definition}[theorem]{Definition}

\newtheorem{example}[theorem]{Example}
\newtheorem{fact}[theorem]{Fact}

\newcommand{\kmax}{k_{\mathrm{max}}}
\newcommand{\kmin}{k_{\mathrm{min}}}
\newcommand{\summ}{\displaystyle\sum}
\newcommand{\prodd}{\displaystyle\prod}
\newcommand{\F}{\mathbb{F}}
\newcommand{\R}{\mathbb{R}}
\newcommand{\mon}{\mathcal{M}}
\newcommand{\K}{\mathcal{K}}

\newcommand{\supker}{\mathsf{suppker}}
\newcommand{\supp}{\mathsf{supp}}
\newcommand*{\defeq}{\stackrel{\text{def}}{=}}

\begin{document}
\title{A Cauchy-Davenport theorem for linear maps}
\author{Simao Herdade\thanks{Department of Mathematics, Rutgers University.  {\tt simaoh@math.rutgers.edu}.} \and 
John Kim\thanks{Department of Mathematics, Rutgers University.  Research supported in part by NSF Grant Number DGE-1433187. {\tt jonykim@math.rutgers.edu }.} \and
Swastik Kopparty\thanks{Department of Mathematics \& Department of Computer Science, Rutgers University. Research supported in part by a Sloan
Fellowship and NSF grant CCF-1253886. {\tt swastik@math.rutgers.edu }.}}
\maketitle

\begin{abstract} 
We prove a version of the Cauchy-Davenport theorem for general linear maps.
For subsets $A,B$ of the finite field $\F_p$,
the classical Cauchy-Davenport theorem gives a lower bound
for the size of the sumset $A+B$ in terms of the sizes of the sets $A$ and $B$.
Our theorem considers a general linear map $L: \F_p^n \to \F_p^m$, and subsets $A_1, \ldots, A_n \subseteq \F_p$,
and gives a lower bound on the size of $L(A_1 \times A_2 \times \ldots \times A_n)$ in terms of the sizes of the sets $A_1, \ldots, A_n$.

Our proof uses Alon's Combinatorial Nullstellensatz and a variation
of the polynomial method.
\end{abstract}

\section{Introduction}

Let $p$ be a prime, and let $\F_p$ denote the finite field of integers modulo $p$.
The classical Cauchy-Davenport theorem states that if $A, B \subseteq \F_p$, 
then the sumset $A + B$ (defined to equal $\{a + b \mid a \in A, b \in B \}$)
satisfies the inequality: $ |A + B| \geq |A| + |B| - 1$, provided $p \geq |A| + |B| - 1$.
It is instructive to compare this with the elementary inequality  $|A + B| \geq |A| + |B| - 1$
for $A , B \subseteq \R$ (this has a simple proof using the natural order on $\mathbb R$).
The Cauchy-Davenport theorem says that this inequality continues to hold mod $p$,
for $p$ large enough.


The Cauchy-Davenport theorem can be seen as a statement about the size of the image of the product set $A \times B$ under the the map $+ : \F_p \times \F_p \to \F_p$. Here we study a similar phenomenon for general linear maps.
Let $L : \F_p^n \to \F_p^m$ be an $\F_p$-linear map. 
For subsets $A_1, \ldots, A_n \subseteq \F_p$, we define
$$L(A_1, \ldots, A_n) = \{ L(a_1, \ldots, a_n) \mid a_i \in A_i \mbox{ for each $i$} \}.$$
(Equivalently, this is the image of $A_1 \times A_2 \times \ldots \times A_n$ under $L$.)
We are interested in a Cauchy-Davenport theorem for $L$: given integers $k_1, \ldots, k_n$, what is the minimum possible size, over subsets $A_i \subseteq \F_p$ with $|A_i| = k_i$, of $|L(A_1, \ldots, A_n)|$?
This question is already interesting for the map $L^*: \F_p^3 \to \F_p^2$, given by $L(x,y,z) = (x+y, x+z)$.

Our main theorem, Theorem~\ref{thm:main}, gives a lower bound on the size of $L(A_1,\ldots, A_n)$. For now
we just state an interesting special case of this theorem, where all the $|A_i| = k$.
While the bound itself is quite complex, the bound (surprisingly) turns out
to be tight for every linear map $L$ when $m = n-1$.

\begin{theorem}
Let $m < n$, and let $L : \F_p^n \to \F_p^m$ be a linear map with rank $m$.
Let $v$ be a nonzero vector in $\ker(L)$ with minimal support, and let $s$ be the size of its support.
Let $k$ be an integer with $p \geq 2k-1$.

Then for every $A_1, \ldots, A_n \subseteq \F_p$, with $|A_i| = k$ for all $i \leq n$, we have:
$$ |L(A_1, \ldots, A_n)| \geq \left(k^{s} - (k-1)^{s} \right) \cdot k^{m-s+1}.$$
\end{theorem}

Some remarks about this theorem:
\begin{itemize}
\item If $m = n -1$ and $p\geq 2k-1$, this lower bound is optimal for {\em every} linear map $L$.
See Lemma~\ref{lem:realtightness}.

If $m = n-1$ and $p < 2k-1$, this lower bound can be violated for {\em every} linear map $L$.

\item If our sets are taken to be subsets of $\mathbb R$ instead of $\F_p$, then 
for $m = n-1$, an identical lower bound holds for every linear map $L: \mathbb R^n \to \mathbb R^m$, 
and this lower bound is optimal for every $L$. As in the case of the Cauchy Davenport theorem, the 
lower bound also has an elementary proof using the natural order on $\mathbb R$.

\item If $m$ is small, and $k$ is large, then the lower bound is approximately $s \cdot k^{m}$.

\end{itemize}

Thus for the map $L^* : \F_p^3 \to \F_p^2$ mentioned above, if $p \geq 2k-1$,
then for every three sets $A_1, A_2, A_3$ with $|A_i| = k$,
we get that
$$|L^*(A_1, A_2, A_3)| \geq k^3 - (k-1)^3 = 3k^2 - 3k + 1,$$
and this is the best bound possible in term of $k$.

\subsection{Proof Outline}

Our proof is based on the Combinatorial Nullstellensatz~\cite{Alon},
generalizing one of the known proofs of the Cauchy-Davenport theorem.

The Combinatorial Nullstellensatz is an algebraic statement characterizing 
multivariate polynomials $Q(Y_1, \ldots, Y_n)$ which vanish on a given product set
$A_1 \times \ldots \times A_n$ as those polynomials which lie in a certain
explicitly given ideal. Let us recall the Combinatorial Nullstellensatz proof~\cite{ANR, Alon} of the
Cauchy-Davenport theorem. For given sets $A_1, A_2 \subseteq \F_p$, 
one wants to prove a lower bound on the size of the sumset $C = A_1+A_2$. Suppose
$C$ was small. The
key step of this proof is to consider the univariate polynomial
$T(X) \in \F_p[X]$, given by:
$$ T(X) = \prod_{c \in C} (X - c),$$
and the bivariate polynomial $Q(Y_1, Y_2) \in \F_p[Y_1, Y_2]$ given by:
$$Q(Y_1, Y_2) = T(Y_1 + Y_2) = \prod_{c \in C}  (Y_1 + Y_2 - c).$$
Since $C$ is small, $T$ and $Q$ are of low degree.
By design, the polynomial $Q$ vanishes on every point $(a_1, a_2) \in A_1 \times A_2$. Thus, by the Combinatorial Nullstellensatz, one concludes that
$Q(Y_1, Y_2)$ must lie in a certain ideal.
Then, inspecting monomials and using the upper-triangular criterion for
linear independence, one shows that no low-degree polynomial of the form $R(Y_1 + Y_2)$ (with $R(X) \in \F_p[X]$)
can lie this ideal. Since $Q(Y_1,Y_2) = T(Y_1 + Y_2)$, this a contradiction.

Our proof will follow the same high-level strategy, but with some important differences.
If $L(A_1, \ldots, A_n)$ is small, we will find a multivariate polynomial $Q$ of low ``complexity''
which vanishes on $A_1 \times A_2 \times \ldots \times A_n$, and thus by the Combinatorial Nullstellensatz, it must lie in a certain ideal $I$.
We then use some linear algebra arguments,
along with the low complexity of $Q$, to show that $Q$ cannot lie in $I$, thus deriving a contradiction.

There are two new technical ingredients that enter the proof.
The first ingredient appears in the construction of the polynomial $Q$. Since the range of $L$
is a high-dimensional vector space, there is no natural way of explictly
giving a polynomial vanishing on $C = L(A_1, \ldots, A_n)$. Instead, we will use a dimension argument 
to show the existence of a suitable polynomial $T(X_1, \ldots, X_m)$ vanishing on $C$,
and define $Q(Y_1, \ldots, Y_n)$ to be $T(L(Y_1, \ldots, Y_n))$.
The second ingredient appears in the linear algebra argument showing that $Q$ does not lie
in $I$. In order to make this argument, we will need $Q$ to have a very special
kind of monomial structure. This monomial structure is enforced when we choose $T$; it is
because of this requirement that we do not simply take $T$ to be a low-degree polynomial,
but instead choose $T$ from a larger space of polynomials satisfying some constraints (this
is what we have termed low complexity in the above description).

\paragraph{Organization of this paper} In the next section we give a formal statement of our main result.
In Section~\ref{sec:proof} we prove our main result. In Section~\ref{sec:example} we discuss limitations of
our methods to prove an optimal bound in the $m < n-1$ case. We conclude with some open problems.

\paragraph{Notation}
We use $[n]$ to denote the set $\{1, 2, \ldots, n\}$. 
For a vector $v \in \F^n$, we define its
support, denoted $\supp(v)$ to be the set of its nonzero coordinates,
namely $\{ i \in [n] \mid v_i \neq 0\}$.
We use $\deg(h)$ to denote the total degree of a polynomial $h$,
and $\deg_{Y}(h)$ to denote the degree in the variable $Y$ of the
polynomial $h$. We say a monomial $\mon$ {\em appears} in a polynomial
$h$ if in the standard representation of $h$ as a linear combination of
monomials, $\mon$ has a nonzero coefficient.


\section{The main result}

We first state our main theorem.
It gives, for every linear map $L: \F_p^n \to \F_p^{m}$,
a lower bound on the size of $L(A_1, \ldots, A_n)$, in terms of the sizes
of $A_1, \ldots, A_n$.


\begin{definition}
For a linear map $L: \F_p^n \to \F_p^m$, we define
the support-kernel of $L$ to be the set:
$$ \supker(L) = \{ S \subseteq [n] \mid  \exists v \in \ker(L), v \neq 0, \mbox{ with }
\supp(v) = S\}.$$
\end{definition}

\begin{theorem}
\label{thm:main}
Let $p$ be prime.
Let $n \geq 2$ be an integer. Let $m < n$.

Let $L: \F_p^n \to \F_p^m$ be a linear map of rank $m$.
Let $S$ be a minimal element of $\supker(L)$.
Let $S'$ be a maximal subset of $[n] \setminus S$
such that $2^{S' \cup S} \cap \supker(L) = \{ S \}$.

Let $1 \leq k_1, \ldots, k_n \leq p$.
Let $\kmax = \max_{i\in S} k_i$ and $\kmin = \min_{i \in S} k_i$.
Suppose $p \geq \kmax + \kmin - 1$.


Define 
$$\lambda =  \left(\left(\prod_{i\in S} k_i\right) - \left(\prod_{i\in S} (k_i-1)\right) \right) \cdot \left(\prod_{i \in S'} k_i \right).$$

Then for every $A_1, \ldots, A_n \subseteq \F_p$ with $|A_i| = k_i$,
 we have:
$$ | L(A_1, \ldots, A_n)| \geq \lambda.$$
\end{theorem}

Taking all the $k_i$ to equal $k$, and observing that $S'$ has size $m+1-s$, we get the theorem stated in the introduction.

The following lemma shows that when $m = n-1$, and $|A_1| = |A_2| = \ldots = |A_n|$,
then the above lower bound is the best possible.
\begin{lemma}
\label{lem:realtightness}
Let $p$ be prime.
Let $n \geq 2$ be an integer. Let $m = n-1$.

Let $L: \F_p^n \to \F_p^m$ be a linear map of rank $m$.
Let $S \subseteq [n]$ be the unique element of $\supker(L)$.
Let $S' = [n] \setminus S$, and observe that $2^{S' \cup S} \cap \supker(L) = S$.

Let $k_1 = k_2 = \ldots =  k_n = k$.


Define 
$$\lambda =  \left(\left(\prod_{i\in S} k_i\right) - \left(\prod_{i\in S} (k_i-1)\right) \right) \cdot \left(\prod_{i \in S'} k_i\right).$$

Then:
\begin{enumerate}
\item If $p \geq 2k-1$,
there exist $A_1, \ldots, A_n \subseteq \F_p$ with $|A_i| = k_i$,
 such that:
$$ | L(A_1, \ldots, A_n)| = \lambda.$$
\item If $p < 2k-1$,
there exist $A_1, \ldots, A_n \subseteq \F_p$ with $|A_i| = k_i$,
 such that:
$$ | L(A_1, \ldots, A_n)| < \lambda.$$
\end{enumerate}
\end{lemma}

\section{Proof of the main theorem}
\label{sec:proof}
For a linear map $L : \F_p^n \to \F_p^m$ and integers $k_1, \ldots, k_n$,
define:
$$ \mu(L, k_1, \ldots, k_n) \defeq \min_{\substack{A_1, A_2, \ldots, A_n \subseteq \F_p \\ |A_i| = k_i}} |L(A_1, \ldots, A_n)|.$$

The proof of the main theorem, Theorem~\ref{thm:main} has two steps.
The first step performs elementary operations on the linear map $L$
to bring it into a simple form, while preserving the value of $\mu(L, k_1, \ldots, k_n)$.   
The second step applies the polynomial method to give a lower bound on $\mu(L, k_1, \ldots, k_n)$
for these simple $L$.  
The allowable operations to simplify the linear map are listed in Lemma~\ref{lem:transform} 
and the lower bound for the simpler map is the subject of Theorem~\ref{thm:GenCD}.

\begin{lemma}
\label{lem:transform}
Let $L: \F_p^n \to \F_p^m$ be a linear map, and let $1 \leq k_1, \ldots, k_n \leq p$.
\begin{enumerate}
\item
Let $L' : \F_p^m \to \F_p^m$ be a full rank linear transformation.
Then $\mu(L, k_1, \ldots, k_n) = \mu(L' \circ L, k_1, \ldots, k_n)$.

\item 
Let $L'' : \F_p^n \to \F_p^n$ be a linear map whose matrix is a
diagonal matrix with all diagonal entries nonzero.
Then $\mu(L , k_1, \ldots, k_n) = \mu(L\circ L'', k_1, \ldots, k_n)$.

\item
Let $\pi: [n] \to [n]$ be a permutation.
Let $L_\pi: \F_p^n \to \F_p^n$ be the linear map that permutes coordinates according
to $\pi$ (i.e.; $L_\pi(e_i) = e_{\pi(i)}$).
Then $\mu(L , k_1, \ldots, k_n) = \mu(L\circ L_{\pi}, k_{\pi^{-1}(1)}, \ldots, k_{\pi^{-1}(n)})$.

\end{enumerate}
\end{lemma}

\begin{proof}
\begin{enumerate}
\item
$L'$ is an isomorphism, so 
$$|L'\circ L(A_1, \ldots, A_n)| = |L(A_1, \ldots, A_n)|.$$
Taking the minimum over the choices of the sets $A_i, i\in [n]$, we get $\mu(L, k_1, \ldots, k_n) = \mu(L' \circ L, k_1, \ldots, k_n)$.

\item
Applying $L''$ to $(A_1,\ldots,A_n)$ simply scales the set $A_i$ by a factor of $L''_{i,i}$.  In particular, $L''$ preserves the sizes 
of the sets.  So we have:
$$|L\circ L''(A_1, \ldots, A_n)| = |L(L''_{1,1}A_1, \ldots, L''_{n,n}A_n)| \geq \mu(L, k_1, \ldots, k_n).$$
Taking the minimum over the choices of the sets $A_i, i\in [n]$, we get $\mu(L\circ L'', k_1, \ldots, k_n) \geq \mu(L' \circ L, k_1, \ldots, k_n)$.  

For the other direction, observe that any scaling is reversible by an inverse scaling:
$$|L(A_1, \ldots, A_n)| = |L\circ L''(\frac{1}{L''_{1,1}}A_1, \ldots, \frac{1}{L''_{n,n}}A_n)| \geq \mu(L\circ L'', k_1, \ldots, k_n).$$
Taking the minimum over the $A_i, i\in [n]$ gives the reverse inequality.

\item
$L_\pi$ permutes the indices of the sets, and so permutes the sizes of the sets.  Taking this into account, the size of the image 
should remain the same:
\begin{eqnarray*}
|L\circ L_\pi(A_{\pi^{-1}(1)}, \ldots, A_{\pi^{-1}(n)})| &=& |L(A_1, \ldots, A_n)| \geq \mu(L, k_1, \ldots, k_n),\\
|L(A_1, \ldots, A_n)| &=& |L\circ L_\pi(A_{\pi^{-1}(1)}, \ldots, A_{\pi^{-1}(n)})| \\
&\geq & \mu(L\circ L_{\pi}, k_{\pi^{-1}(1)}, \ldots, k_{\pi^{-1}(n)}).
\end{eqnarray*}
Taking the minimum over the $A_i, i\in [n]$ gives both directions of the inequality.

\end{enumerate}
\end{proof}


\begin{theorem}
\label{thm:GenCD}
Let $p$ be prime. Let $m \geq 1$ be an integer.

Let $U_1, U_2, \ldots, U_{m}, V \subseteq F_p$ 
be subsets of size $|U_i| = k_i$ for $1\leq i\leq m$, and $|V| = \hat{k}$.
Suppose $p \geq \hat{k} + k_i-1$ for each $i$.

Let
$$C = \{(u_1+v,u_2+v,\ldots,u_m+v)|u_i\in U_i \text{ for each $i$}, v \in V \}.$$  

Then
$$|C| \geq \hat{k} \cdot \prodd_{i=1}^{m}{k_i}  - (\hat{k}-1) \cdot \prodd_{i=1}^{m}(k_i-1).$$
\end{theorem}


\subsection{Preliminaries: multivariate polynomials and Combinatorial Nullstellensatz}

In preparation for our proof of Theorem~\ref{thm:GenCD}, we recall the statement
of the Combinatorial Nullstellensatz, along with some important facts about
reducing multivariate polynomials modulo ideals of the kind that arise in the Combinatorial Nullstellensatz.

\begin{lemma}[Combinatorial Nullstellensatz~\cite{Alon}]
Let $\F$ be a field, and let $A_1, \ldots, A_n \subseteq \F$.
For $i \in [n]$, let $P_i(T)\in \F[T]$ be given by
$P_i(T) = \prodd_{\alpha \in A_i} (T - \alpha)$.

Let $h(Y_1, \ldots, Y_n) \in \F[Y_1, \ldots, Y_n]$.
Then $h(Y_1, \ldots, Y_n)$ vanishes on $A_1 \times \ldots \times A_n$
if and only if $h$ lies in the ideal generated by $P_1(Y_1), P_2(Y_2), \ldots, P_n(Y_n)$.
\end{lemma}

Now let $P_1(T), \ldots, P_n(T) \in \F[T]$ be polynomials, with
$\deg(P_i) = k_i$.
Let $I$ be the ideal generated by $\langle P_i(Y_i)\rangle_{ i \in [n]}$.

Given this setup, we now discuss the operation of {\em reducing a polynomial mod $I$}. A monomial $\prod_{i = 1}^n Y_i^{e_i}$ is called
{\em legal} for $I$ if $e_i < k_i$ for each $i \in [n]$.
Given a polynomial $h$, there is a canonical reduction mod $I$, denoted $\overline{h}$, with the property that $h \equiv \overline{h} \mod I$,
and that every monomial appearing in the expansion of $\overline{h}$ is
legal for $I$ (equivalently, for each $i$ we have
$\deg_{Y_i}(\overline{h}) < k_i$).
This canonical reduction can be obtained as follows.
Reducing a polynomial mod $P_i(Y_i) = Y_i^{k_i} - \sum_{j=0}^{k_i-1} a_j Y_i^j$
is simply the act of repeatedly replacing every occurrence of $Y_i^{k_i}$
with $\sum_{j=0}^{k_i-1} a_j Y_i^{j}$, until the $Y_i$ degree is less
than $k_i$. Reducing the polynomial $h$ mod $P_i(Y_i)$ in succession
for each $i \in [n]$ gives the canonical reduction $\overline{h}$.

Here are some important (and easy to verify) points about canonical reduction:
\begin{enumerate}
\item $h \in I$ if and only if $\overline{h} = 0$.
\item The map $h \mapsto \overline{h}$ is $\F$-linear.
\end{enumerate}

It will be important for us to understand the degrees of the monomials in $\overline{h}$.  
Let $\mon = \prodd_{i=1}^{n} Y_i^{e_i}$ be a monomial, and consider its reduction $\overline{\mon}$ mod $I$.  
If $e_i<k_i$ for each $i\in [n]$, then we have $\overline{\mon} = \mon$.  Furthermore, if there is some $e_i\geq k_i$, then 
$\deg\overline{\mon}<\deg(\mon)$.  This is because the act of replacing $Y_i^{k_i}$ with a lower degree polynomial in $Y_i$ strictly decreases the degree.  
Combining these two facts, we get the following fact.

\begin{fact}
\label{fact}
With notation as above, let $h(Y_1, \ldots, Y_n) \in \F[Y_1, \ldots, Y_n]$.
Suppose $\mon$ is a monomial that (1) appears in $h$, (2) has $\deg(\mon) = \deg(h)$, and (3) is legal for $I$.

Then $\mon$ appears in the canonical reduction $\overline{h}$. 
\end{fact}
This is because $\overline{\mon} = \mon$, and 
the canonical reductions of the other monomials will have smaller degree than $\mon$, and will therefore leave $\mon$ untouched.

Very similar considerations give us the following related fact.
\begin{fact}
\label{fact2}
With notation as above, let $h(Y_1, \ldots, Y_n) \in \F[Y_1, \ldots, Y_n]$.
Suppose $\mon$ is a monomial that (1) appears in $\overline{h}$, (2) has $\deg(\mon) = \deg(h)$, and (3) is legal for $I$.

Then $\mon$ appears in $h$. 
\end{fact}

\subsection{Correlated sumsets and the polynomial method}

We now prove Theorem~\ref{thm:GenCD}.

\begin{proof}
We begin by defining some sets of monomials which will be useful to us.

In the polynomial ring $\F_p[Y_1, \ldots, Y_m, Z]$, consider the 
following set of monomials:

\begin{multline*}
\Gamma = \{Y_1^{e_1}Y_2^{e_2}\cdots Y_m^{e_m} Z^e|0\leq e_i\leq k_i-1 \text{ for each $i$, and } 0 \leq e \leq \hat{k}-1,  \\
 \text{and $e > 0$} \Rightarrow e_i = k_i-1 \text{ for some } i\}.
\end{multline*}

We will also consider the polynomial ring $\F_p[X_1, \ldots, X_m]$.
To each monomial $\mon(Y_1, \ldots, Y_m, Z) \in \Gamma$, we associate 
a monomial $\phi(\mon) \in \F_p[X_1, \ldots, X_m]$ as follows. If $\mon(Y_1, \ldots, Y_m,Z) = Y_1^{e_1}Y_2^{e_2}\cdots Y_m^{e_m}Z^e$, 
then define:
$$ \phi(\mon) =
\begin{cases}
\prodd_{i=1}^m X_i^{e_i} & \text{if } e = 0 \\
\left(\prodd_{i=1}^m X_i^{e_i} \right) \cdot X_j^{e}  & \text{if } e > 0, \text{ where } j\text{ is the first index}\\
& \text{ so that } e_j = k_j-1.
\end{cases}$$


Let  $\Delta = \{ \phi(\mon) \mid \mon \in \Gamma\}$ be the set of all such monomials constructed in this way.

Note that $\phi$ is a bijection, and $\phi$ preserves the degree of each monomial.  Thus, $\phi$ also gives a bijection 
when we restrict to monomials in $\Gamma$ and $\Delta$ of fixed total degree.  We defined $\phi$ so that $\phi^{-1}$ 
would have the following description:  Let $X_1^{f_1}X_2^{f_2}\cdots X_{m}^{f_{m}}\in \Delta$.  Let $e_i = \min\{f_{i},k_i-1\}$ for each $i \in [m]$.  Let $e = \summ_{i=1}^{m}{f_i} - \summ_{i=1}^{m}{e_i}$.
Then 
$$\phi^{-1}(X_1^{f_1}X_2^{f_2}\cdots X_{m}^{f_{m}}) = Y_1^{e_1}Y_2^{e_2}\cdots Y_m^{e_m} \cdot Z^e.$$
Note that by choice of $e$, $\phi^{-1}$ preserves degree.

With these definitions in hand, we proceed with the main parts of the proof.

\subsubsection*{Interpolating a polynomial} 
Suppose for contradiction that $|C| < \hat{k} \cdot \left( \prodd_{i=1}^{m}{k_i}\right) - (\hat{k} - 1) \cdot \left(\prodd_{i=1}^{m}(k_i-1) \right)$.  
Since $|\Delta| = |\Gamma| = \hat{k} \cdot \prodd_{i=1}^{m}{k_i} -(\hat{k} - 1) \cdot \prodd_{i=1}^{m}(k_i-1)$, there is a non-zero polynomial 
$f(X_1,\ldots,X_{m}) = \summ_{\K \in \Delta}{c_\K \K(X_1, \ldots, X_m)}$ which vanishes on $C$.  By the definition of $C$, this means that 
$g(Y_1,\ldots,Y_m, Z) \defeq f(Y_1+Z,\ldots, Y_m+Z)$ is a non-zero polynomial vanishing on every point
$(u_1,u_2, \ldots, u_m, v) \in \prodd_{i=1}^m U_i \times V$.

\subsubsection*{Application of the Combinatorial Nullstellensatz}
For each $1\leq i\leq m$, let $P_i(Y_i) = \prodd_{a\in U_i}(Y_i-a)$. 
Also let $P(Z) = \prodd_{a \in V} (Z-a)$.

By the Combinatorial Nullstellensatz,
$$g(Y_1,\ldots, Y_m, Z) \equiv 0 \pmod{ I},$$
where $I$ is the ideal generated by the $P_i(Y_i), i\in [m]$ and $P(Z)$.

%

Explicitly, we have that:
$$\sum_{\K \in \Delta} c_{\K}\K(Y_1+Z, Y_2 +Z, \ldots, Y_m + Z) \equiv 0 \pmod{I},$$
where at least one $c_\K$ is nonzero.

Consider the canonical reduction $\overline{g}$ of $g \mod I$: since $g \in I$
we get that $\overline{g} = 0$.
On the other hand, we have by linearity of canonical reduction:
$$\overline{g} = \sum_{\K \in \Delta} c_{\K}\overline{\K}(Y_1, Y_2, \ldots, Y_m, Z),$$
where $\overline{\K}(Y_1, Y_2, \ldots, Y_m, Z)$ is the canonical reduction mod $I$ of $\K(Y_1+Z, Y_2 +Z, \ldots, Y_m + Z)$.  
By Fact~\ref{fact}, any monomial $\mon$ that appears in the expansion of $\K(Y_1+Z, Y_2 +Z, \ldots, Y_m + Z)$ with $\deg(\mon) = \deg(\K)$ and is legal for $I$, also appears in $\overline{\K}(Y_1, Y_2, \ldots, Y_m, Z)$.

\subsubsection*{Arriving at a contradiction}
We may now summarize the strategy for the rest of the proof.
We will first find an ordering of the monomials in $\Delta$ such that:
\begin{enumerate}
\item If $\K, \K'$ are monomials in $\Delta$ with $\deg(\K') < \deg(\K)$,
then $\K'$ is smaller than $\K$ in the ordering.
\item For each $\K \in \Delta$, there is some monomial $\mon_\K(Y_1, \ldots, Y_m, Z)$ with the following four properties:
\begin{enumerate}
\item $\mon_{\K}$ appears the expansion of  $\K(Y_1+Z, Y_2 +Z, \ldots, Y_m + Z)$,
\item $\deg(\mon_{\K}) = \deg(\K)$,
\item $\mon_{\K}$ is legal for $I$,
\item $\mon_{\K}$ does not appear in the expansion of $\K'(Y_1+Z, Y_2+Z, \ldots, Y_m+Z)$ for any $\K' \in \Delta$ smaller than $\K$ in the ordering.
\end{enumerate}
\end{enumerate}

Once we have such an ordering, consider the largest $\K$ in the ordering
for which $c_{\K} \neq 0$. By Fact~\ref{fact}, $\mon_{\K}$ appears in $\overline{\K}(Y_1, \ldots, Y_m, Z)$. For every other $\K' \in \Delta$ with $c_{\K'} \neq 0$, we will show that $\overline{\K'}(Y_1, \ldots, Y_m, Z)$ does not include the monomial $\mon_{\K}$; this then shows that $\mon_{\K}$ appears in $\overline{g}$ with a nonzero coefficient, contradicting our equation $\overline{g} = 0$.
This gives the desired contradiction.

\subsubsection*{Monomial $\mon_{\K}$ does not appear in $\overline{\K'}$ (for $\K' \neq \K$ with $c_{\K'} \neq 0$)}
Suppose $\K' \in \Delta$, $\K' \neq \K$ and $c_{\K'} \neq 0$.
We will show that $\mon_{\K}$ does not appear in $\overline{\K'}$.
By choice of $\K$, we have that $\K'$ is smaller than $\K$ in 
the ordering, and hence that $\deg(\K') \leq \deg(\K)$.


Suppose $\mon_{\K}$ appeared in $\overline{\K'}$. Then the following
chain of inequalities:
$$\deg(\mon_{\K}) \leq \deg(\overline{\K'}) \leq \deg(\K'(Y_1 + Z, \ldots, Y_m +Z )
\leq \deg(\K') \leq  \deg(\K) = \deg(\mon_{\K}),$$
(because of the equality of the endpoints, this is a chain of equalities), shows that $\deg(\mon_{\K}) = \deg(\K'(Y_1 +Z, \ldots, Y_m+Z))$.
Thus by Fact~\ref{fact2}, we can conclude that $\mon_{\K}$ appears
in $\K'(Y_1+Z, \ldots, Y_m + Z)$. But this contradicts the 
property that $\mon_{\K}$ does not appear in $\K'(Y_1 + Z, \ldots, Y_m+Z)$
for any $\K' \in \Delta$ that is smaller than $\K$ in the ordering.
Thus $\mon_{\K}$ cannot appear in $\overline{\K'}$.






\subsubsection*{The ordering of $\Delta$}


All that remains now is to define the ordering of $\Delta$, and to prove 
the desired properties of this ordering.

Arrange the monomials in $\Delta$ in order of increasing total degree.  Within each fixed total degree, order by decreasing 
$\deg_{Z}(\phi^{-1}(\K(X_1, \ldots, X_m)))$.  Then for $\K(X_1, \ldots, X_m) \in \Delta$ in that ordering, set 
$\mon_{\K} = \phi^{-1}(\K(X_1,\ldots,X_m)) = Y_1^{e_1}Y_2^{e_2}\cdots Y_m^{e_m}Z^e$.  We claim that $\mon_{\K}$ satisfies 
the four properties listed above.

\begin{enumerate}[(a)]
\item \emph{$\mon_{\K}$ appears the expansion of  $\K(Y_1+Z, Y_2 +Z, \ldots, Y_m + Z)$}:

We show that the coefficient of $\mon_{\K}$ in $\K(Y_1+Z, Y_2 +Z, \ldots, Y_m + Z)$ is non-zero.  
By the definition of $\Delta$, $\deg_{X_i}(\K)\leq \hat{k}+k_i-2 < p$ for $i\in [m], \K\in \Delta$.  Also, there is at most 
one $i$ such that $\deg_{X_i}(\K) > k_i-1$.  Call this index $j$ if it exists, and let $l = \deg_{X_j}(\K)$.  
Since $\phi^{-1}(\K)$ extracts the largest powers of $Y_i$ in 
$\K(Y_1+Z, Y_2 +Z, \ldots, Y_m + Z)$ up to $k_i-1$ for $i\in [m]$, 
we get that the coefficient of $\mon_{\K}$ is $1$ if $j$ does not exist and 
$\binom{l}{k_j-1}$ if $j$ exists.  In both cases, the coefficient of $\mon_{\K}$ is non-zero in $F_p$ as $l < p$.

\item \emph{$\deg(\mon_{\K}) = \deg(\K)$}:

Recall that $\phi$ is a bijection from one set of monomials to another which preserves the degree of the monomials.  
So $\deg(\mon_{\K}) = \deg{\phi^{-1}(\K(X_1,\ldots,X_m))} = \deg(\K)$.

\item \emph{$\mon_{\K}$ is legal for $I$}:

Recall that writing $\K = X_1^{f_1}X_2^{f_2}\cdots X_{m}^{f_{m}}$, we have
$$\phi^{-1}(\K) = Y_1^{e_1}Y_2^{e_2}\cdots Y_m^{e_m} \cdot Z^e,$$
where $e_i = \min\{f_{i},k_i-1\}$ for each $i \in [m]$, and $e = \summ_{i=1}^{m}{f_i} - \summ_{i=1}^{m}{e_i}$.
So $e_i \leq k_i-1$, $\forall i\in [m]$.  It remains to show that $e \leq \hat{k}-1$.  
Suppose $f_i\leq k_i-1$, $\forall i \in [m]$.  Then $e_i = f_i$, $\forall i\in [m]$ and so $e=0$.  
Otherwise, $f_i\leq k_i-1$ for all but one $i\in [m]$, call this index $j$.  
We have $e_i = f_i$ for $i\neq j$ and $e_j = k_j-1$.  So 
$$e = f_j-e_j \leq \hat{k}+k_j-2 - (k_j-1) = \hat{k}-1.$$

\item \emph{$\mon_{\K}$ does not appear in the expansion of $\K'(Y_1+Z, Y_2+Z, \ldots, Y_m+Z)$ for any $\K' \in \Delta$ smaller than $\K$ in the ordering}:

To show that the monomials selected by $\phi^{-1}$ do not appear in any previous entries of the ordering, 
first note that the degree of $\mon_{\K}$ is too large to have appeared in any previous $\K'\in \Delta$ of lower total degree.  
Next, consider the expansion of a previous $\K'(Y_1+Z, Y_2 +Z, \ldots, Y_m + Z)$ in the ordering of the same total degree, 
then $\mon_{\K'} = \phi^{-1}(\K'(X_1,\ldots,X_m)) = Y_1^{e_1'}Y_2^{e_2'}\cdots Y_m^{e_n'}Z^{e'}$ must have $e_i'<e_i\leq k_i-1$ for some 
$i\in [m]$, as $e'\geq e$.  By the way $\phi$ is defined, this means that $\deg_{X_i}(\K') = e_i'$, so 
$\deg_{Y_i}(\K'(Y_1+Z, Y_2 +Z, \ldots, Y_m + Z)) = e_i'$.  But $\deg_{Y_i}(\mon_{\K}) = e_i > e_i'$.  
So $\mon_{\K}$ cannot be a monomial in the expansion of $\K'(Y_1+Z, Y_2 +Z, \ldots, Y_m + Z)$.

\end{enumerate}


This completes the proof that the ordering of $\Delta$ has the desired properties, and hence we arrive at a contradiction.

Thus we must have that $|C| \geq \hat{k} \cdot \left( \prodd_{i=1}^{m}{k_i}\right) - (\hat{k} - 1) \cdot \left(\prodd_{i=1}^{m}(k_i-1) \right)$.
\end{proof}

\subsection{Proving the main result}

We now combine Lemma~\ref{lem:transform} and Theorem~\ref{thm:GenCD} to prove our main theorem, Theorem~\ref{thm:main}.

\medskip
\noindent{\bf Proof of Theorem~\ref{thm:main}: }

By basic linear algebra, we have that $|S| \leq m+1$, and $|S \cup S'| = m+1$.

We first get rid of the coordinates in $[n] \setminus (S \cup S')$.  Observe that taking away elements from any of the sets 
$A_i$ cannot increase the size of the image $|L(A_1, \ldots, A_n)|$.
Let $\textbf{a}\in\prodd_{i\in [n]\setminus (S\cup S')}{A_i}$.  
Fix the coordinates in $[n] \setminus (S \cup S')$ to $\textbf{a}$ and consider the resulting
map $M: \F_p^{m+1} \to \F_p^m$ (i.e., $M(\mathbf{x}) = L(\mathbf{x}, \mathbf{a})$).
If $L':\F_p^{m+1} \to \F_p^m$ is the linear map 
obtained by restricting the coordinates of $[n] \setminus (S \cup S')$ to $0$,
then the image $L'\left(\prodd_{i\in S\cup S'}{A_i}\right)$ is a translate of the 
image of $M\left(\prodd_{i\in S\cup S'}{A_i}\right)$.  So we have:
$$|L(A_1, \ldots, A_n)|\geq \left|M\left(\prodd_{i\in S\cup S'}{A_i}\right)\right| = L'\left(\prodd_{i\in S\cup S'}{A_i}\right).$$

Then a lower bound on $L'\left(\prodd_{i\in S\cup S'}{A_i}\right)$ gives a lower bound on $|L(A_1, \ldots, A_n)|$.

The next step is to use the simple transformations in Lemma~\ref{lem:transform} to greatly simplify our linear map $L'$, 
while preserving $\mu(L,k_1,\ldots,k_n)$.  The transformations allow us to apply elementary row operations on $L'$, 
scale the columns of $L'$, and rearrange the columns of $L'$.

As $L'$ has rank $m$, $\ker(L')$ has rank $1$.  Consider a nonzero vector $v\in\ker(L')$.  
Then $S$ must be the support of $v$.  
Let $\hat{i}$ be the index in $S$ that minimizes $k_i$, i.e. $\hat{i} = \arg\min_{i\in S}{k_i}$.  

With the above row and column operations at our disposal, we perform the following reduction of the problem.  First, 
permute the columns so that the columns with indices in $S$ are on the left and move column $\hat{i}$ so that 
it is the first column.  Then the last $m$ columns are now linearly independent.  
This is because if they were linearly dependent, there would be a nonzero vector in the kernel of $L$ whose support 
does not include $\hat{i}$.  So there would be two nonzero vectors in $\ker(L)$ with different supports, which is impossible.  
Next, apply the sequence of elementary row 
operations that turns the last $m$ columns into the identity matrix.  Scale each row so that the first element is 
either $0$ or $1$.  Finally, scale each of the last $m$ columns so that they again form the identity matrix.  We are 
left with a column of $1$'s and $0$'s followed by the $m$ by $m$ identity matrix.  We will call this matrix $\hat{L'}$, the 
reduction of $L'$.  

$$\hat{L'} = \left(
\begin{array}{c|ccccc}
1 			& &&    && \\
\vdots  & &&   	&& \\
1 			& && I_m&& \\
0 			& &&   	&& \\
\vdots  & &&   	&& \\
0 			& &&   	&& \\
\end{array}
\right).$$

Considering the projection $P$ of the image of $\hat{L'}(\prod_{i\in S} A_i, \prod_{i\in S'} A_i)$ onto the first $|S|-1$ coordinates, 
we find ourselves in the setting of Theorem~\ref{thm:GenCD}.  
Letting $U = A_{\hat{i}}$, and $\{V_1,\ldots,V_{|S|-1}\} = \{A_i\mid i\in S-\{\hat{i}\}\}$, Theorem~\ref{thm:GenCD} tells us that 
$$|P|\geq \prod_{i\in S} k_i - \prod_{i\in S} (k_i-1).$$

Finally, note that as $\textbf{a}'$ varies in the set $\prod_{i\in S'}A_i$, the sets 
$\hat{L'}(U,V_1,\ldots,V_{|S|-1}, \textbf{a}')$ are all translates of $P$ and are
disjoint (the disjointness follows from the fact that $\supker(L) \cap 2^{S \cup S'} = \{S \}$).
Hence, the total size of the image of $\hat{L'}$ is at least $|P| \cdot \prod_{i \in S'} |A_i|$, which
is at least:
$$\left(\left(\prod_{i\in S} k_i\right) - \left(\prod_{i\in S} (k_i-1)\right) \right) \cdot \left(\prod_{i \in S'} k_i \right),$$
as desired.






\hfill\qed

\noindent {\bf Proof of Lemma~\ref{lem:realtightness}: }

We first provide a tight example for our lower bound when $p\geq 2k-1$.  
Using the same transformations as above, we produce the simple linear transformation $\hat{L}$ from $L$.  
Lemma~\ref{lem:transform} implies that providing a tight example for $\hat{L}$ implies the existence of a tight example for $L$.  
We claim that setting $A_i = \{0,\ldots,k_i-1\}$ attains the smallest possible image size   
$\left(\prod_{i\in S} k_i - \prod_{i\in S} (k_i-1) \right) \cdot \prod_{i \in S'} k_i$.

As before, every choice of $\textbf{a}\in \prod_{i\notin S}A_i$ yields $|P|$ distinct points in the image of $\hat{L}$, 
where $P$ is the projection of $\hat{L}(\prod_{i\in S} A_i, \prod_{i\in S'} A_i)$ onto the first $|S|-1$ coordinates.  
So it suffices to show that $|P| \geq \left(\prod_{i\in S} k_i - \prod_{i\in S} (k_i-1) \right)$. 
This is equivalent to showing that equality is attained in Theorem~\ref{thm:GenCD} 
when the sets are all taken to be intervals starting from $0$.

Suppose we have sets $U_i = \{0,\ldots,k_i-1\}, i\in [m]$ and $V=\{0,\ldots,\hat{k}-1\}$.  We want to show that
$C = \{(u_1+v,u_2+v,\ldots,u_m+v)|u_i\in U_i \text{ for each $i\in [m]$}, v \in V \}$ has size exactly equal to 
$\hat{k} \cdot \prodd_{i=1}^{m}{k_i}  - (\hat{k}-1) \cdot \prodd_{i=1}^{m}(k_i-1)$ as long as $p\geq \hat{k}+k_i-1$.  
In particular, this will give a tight example when the set sizes are all the same.

Let $C_j = \{(u_1+j,u_2+j,\ldots,u_m+j)|u_i\in U_i \text{ for each $i\in [m]$}\}, j = 0,\ldots,\hat{k}-1$.  Then 
$C = \bigcup_{j=0}^{\hat{k}-1}C_j$.  
We start with $|C_0| = \prodd_{i=1}^{m}{k_i}$, and ask how many additional elements we add when we take 
the union with $C_1$:

\begin{eqnarray*}
|C_1-C_0| &=& |C_1|-|C_1 \cap C_0|\\
&=& \prodd_{i=1}^{m}{k_i} - \prodd_{i=1}^{m}{(k_i-1)}.
\end{eqnarray*}

Since $p\geq \hat{k}+k_i-1$, none of the sums that we take exceed $p-1$, so we will continue to add 
$\prodd_{i=1}^{m}{k_i} - \prodd_{i=1}^{m}{(k_i-1)}$ for each successive $C_j$.  Total this gives 
$\prodd_{i=1}^{m}{k_i} + (\hat{k}-1) \cdot\left(\prodd_{i=1}^{m}{k_i} - \prodd_{i=1}^{m}{(k_i-1)}\right)$, which is equal to 
$\hat{k} \cdot \prodd_{i=1}^{m}{k_i}  - (\hat{k}-1) \cdot \prodd_{i=1}^{m}(k_i-1)$.

We now show that the lower bound is not tight when $p < 2k-1$.  In fact, the same example of taking the sets to be intervals 
will produce an image whose size is strictly smaller than our lower bound.  
Let $U_i = \{0,\ldots,k-1\}, i\in [m]$ and $V=\{0,\ldots,k-1\}$ in the statement of Theorem~\ref{thm:GenCD}.  
We want to show that $C = \{(u_1+v,u_2+v,\ldots,u_m+v)|u_i\in U_i \text{ for each $i\in [m]$}, v \in V \}$ 
has size strictly less than $k^{m+1} - (k-1)^{m+1}$.  

As before, let $C_j = \{(u_1+j,u_2+j,\ldots,u_m+j)|u_i\in U_i \text{ for each $i\in [m]$}\}, j = 0,\ldots,k-1$.  Then 
$C = \bigcup_{j=0}^{k-1}C_j$.  Note that the element $(k-1+p-k+1,\ldots,k-1+p-k+1,k-1+p-k+1) = (0,\ldots,0) \in C_{p-k+1}$ 
is in $C_0$.  But this was one of the ``new'' elements of $C_{p-k+1}$ that we counted in the argument for the tight example, 
which was previously not in any of the $C_i$, for $i<p-k+1$.  Hence, the number $k^{m+1} - (k-1)^{m+1}$ is a strict overcount 
for the number of elements in the image.

\hfill\qed

\section{Linear maps of smaller rank}
\label{sec:example}

Our lower bound in the general case $n > m-1$ is not tight for every linear map. The main reason for this is that our proof strategy only uses information about the support of vectors in the kernel of $L$ (and not the actual vectors). As the following example shows, if $m < n -1$ the optimal lower bound for
$L(A_1, \ldots, A_n)$ may not be determined solely be the set of all supports of vectors in $\ker(L)$.

\begin{example}
Let $p$ be a large prime, and let $k \ll p$.
Consider the following $2 \times 4$ matrices over $\F_p$:
\begin{align*}
M &= \left[ \begin{array}{cccc} 1 & 0 &\ 2\ & 1 \\ 0 & 1 & 1 &\ 2\ \end{array} \right],  \\
M' &= \left[ \begin{array}{cccc} 1 & 0 & 100 & 1 \\ 0 & 1 & 1 & 100 \end{array} \right].
\end{align*}

Define $L : \F_p^4 \to \F_p^2 $ and $L': \F_p^4 \to \F_p^2$
by $L(x) = Mx$ and $L'(x) = M'x$.
Observe that $\supker(L)$ and $\supker(L')$ both equal ${ [4] \choose \geq 3}$.

Letting $A_1, A_2 , A_3,A_4=   \{1,2,\ldots, k\} \subseteq \F_p$, then  $|L(A_1, A_2, A_3 ,A_4)| \leq 16 k^2$.

In contrast, we will show in Lemma~\ref{lem:100} that $|L'(A_1', A_2',A_3', A_4')| \geq 100 k^2$,  for any $k$-elements sets $A_1', A_2', A_3' , A_4'  \subseteq \F_p$,

\end{example}

Our analysis of this example will use some results on ``sums of dilates''.
For a constant $\lambda$ and a set $A$, we define the {\em dilate} $\lambda A$ denote the set $\{ \lambda a \mid a \in A \}$.
 We will use the following result of Pontiveros \cite{dilates} (which builds on a beautiful result of Bukh~\cite{Bukh})
on sums of dilates in  $\mathbb{Z}_p$. 
\begin{lemma}\label{Pontiveros}
 For every coprime $\lambda_1,\cdots, \lambda_n \in \mathbb{Z}$, there exists a constant $\alpha>0$ such that $|\lambda_1 X+\lambda_2 X+\cdots+ \lambda_n X| \geq\big( \sum \lambda_i \big)\cdot|X| -o(|X|)$, for sufficiently large prime~$p$, and every $X\subseteq \mathbb{Z}_p$, with $|X|\leq \alpha p$.
\end{lemma}

We use this estimate on the size of the sum of dilates, to construct linear maps with arbitrarily large image.   

\begin{lemma}\label{lem:100}
For every positive integer constant $c$,   there is a linear map $L: \F_p^4 \to \F_p^2$
such that for  every  $A_1, A_2, A_3, A_4 \subseteq \F_p$ with $|A_i| = k$, and any prime $p$ sufficiently larger than $k$,
we have:
$$ L(A_1, A_2, A_3, A_4) \geq c k^2.$$

\end{lemma}

\begin{proof}
Consider the linear map  
$$L(A_1,A_2,A_3,A_4) = \{(a_1+c\cdot a_3 +a_4 ,\, a_2 + a_3 + c\cdot a_4)|\, (a_1,a_2,a_3,a_4) \in A_1 \times A_2 \times A_3\times A_4 \}$$ 
and let $A_1,A_2,A_3,A_4 \subseteq F_p$ be any  $k$-elements sets.

	By Ruzsa triangle inequality \cite{TaoandVu}, 
$$| A_4+ c^2A_4 | \leq \frac{ |A_4+c A_3| \,|c A_3+c^2A_4|}{|c A_3|}=  \frac{ |A_4+c A_3| \,|A_3+c A_4|}{|A_3|}$$

	From Lemma \ref{Pontiveros} we know that $| A_4+ c^2 A_4 |\geq c^2 |A_4|$, assuming $p$ sufficiently larger than $k=|A_4|$. 

 Hence  $|A_4+c A_3| \cdot|A_3+c A_4| \geq c^2 |A_4| |A_3| = c^2 k^2$. 

Without loss of generality,  assume that $|A_3+c A_4|\geq c k$.

In particular, fixing $a_2 \in A_2$, and an element $(a_3+c\cdot a_4) \in A_3+c A_4$, the subset $ \{ (a_1+ (c\cdot a_3+a_4), a_2 + a_3 + c\cdot a_4)| a_1 \in A_1\}$  has at least $k$ elements, all with the same second coordinate.
Therefore holding some element $a_2 \in A_2$ fixed, and letting $a_3\in A_3,a_4\in A_4$ be any elements, we obtain $|A_3+c A_4|$ distinct second coordinates, and so  $$|\{ (a_1+ c\cdot a_3+a_4, a_2 + a_3 + c\cdot a_4)\}| a_1 \in A_1, a_3 \in A_3, a_4 \in A_4\}|\geq c k^2$$

We conclude that $|L(A_1,A_2,A_3,A_4)|\geq c k^2$.
\end{proof}

\section{Questions}

We conclude with some interesting open questions.
\begin{enumerate}
\item The main open question is to obtain the best bound for the Cauchy-Davenport problem for every linear map.
\item Even for the case $m = n-1$ and all the $k_i$ equal to $k$,
we do not know the optimal bound for the Cauchy-Davenport problem when $p < 2k-1$. Our method can be extended
to give a better bound, but we believe that this is not the optimal bound.
\item What can be said about the ``symmetric" Cauchy-Davenport problem: what is smallest
possible size of $L(A, A, \ldots, A)$ over all sets $A$ with $|A| = k$?
This seems to be closely related to the theory of sums of dilates.
\item Even over $\mathbb R$, finding the optimal bound for the Cauchy-Davenport problem for every linear map seems nontrivial.
\item It will be interesting to study analogues of other theorems of additive combinatorics
 in the setting of linear maps.
\end{enumerate}

\bibliographystyle{alpha}

\end{document}